\documentclass[11pt]{article}

\usepackage{amsfonts, amssymb}
\usepackage{amsmath}
\usepackage{verbatim}
\usepackage{indentfirst}
\usepackage{color}

\newtheorem{theorem}{Theorem}[section]
\newtheorem{lemma}[theorem]{Lemma}
\newtheorem{corollary}[theorem]{Corollary}
\newtheorem{prop-def}{Proposition-Definition}[section]

\newenvironment{proof}{\trivlist \item[\hskip \labelsep{\it Proof.}]}{\endtrivlist}

\begin{document}
\title{Infinite Dimensional $3$-Lie Algebras and Their Connections to Harish-Chandra Modules}
\date{}
\author{\small Ruipu  Bai\qquad Zhenheng Li\qquad Weidong Wang}
\maketitle
\vspace{-15mm}
\begin{center}
\end{center}
\parskip 2mm

\noindent{\bf Abstract:}
In this paper we construct two kinds of infinite-dimensional $3$-Lie algebras from a given commutative associative algebra, and show that they are all canonical Nambu $3$-Lie algebras. We relate their inner derivation algebras to Witt algebras, and then study the regular representations of these $3$-Lie algebras and the natural representations of the inner derivation algebras. In particular, for the second kind of $3$-Lie algebras, we find that their regular representations are Harish-Chandra modules, and the inner derivation algebras give rise to intermediate series modules of the Witt algebras and contain the smallest full toroidal Lie algebras without center.

\vspace{2mm}

\vspace{2mm}
\noindent{\bf 2010 MR subject classification}: 17B05,~ 17B10, ~17D99

\section{Introduction}

Nambu \cite{Na} initiated implicitly the study of $3$-Lie algebraic structures for generalized Hamiltonian mechanics by proposing the {\it canonical Nambu bracket} for a triple of classical observables on a three dimensional
phase space with coordinates $x, y, z$
\begin{equation}\label{nb}
    [f_1, f_2, f_3]_\partial = \frac {\partial(f_1, f_2, f_3)}{\partial(x, y, z)}
\end{equation}
where the right-hand side is the Jacobian determinant of the vector function $(f_1, f_2, f_3)$.

As a generalization of Hamiltonian mechanics, Nambu mechanics attracts attention from both mathematicians and physicists. In physics, Takhtajan \cite{Ta} described some basic principles of the formalism of Nambu mechanics, and introduced the notion of {\it  Nambu-Poisson manifold}.  Bayen and Flato studied Hamltonian embedding for Nambu mechanics using Dirac's singular formulism \cite{Ba}. Mukunda and Sudarshan investigated the relationship between Nambu and Hamiltonian mechanics \cite{Mu}. Hirayama studied the realization problem of Nambu mechanics \cite{Hi}.

In mathematics, Filippov \cite{Fi} originated systematically the theory of {\it abstract $n$-Lie algebras}, generalizing both Lie algebras and the canonical Nambu bracket to a more abstract stage. In particular, he studied
the $n$-Lie algebra consisting of polynomials of $n$-variables with respect to the canonical Nambu bracket (\ref{nb}).

We call an abstract $3$-Lie algebra a {\it canonical Nambu $3$-Lie algebra} if it is isomorphic to a $3$-Lie algebra whose bracket is defined by the canonical Nambu bracket (\ref{nb}).

A natural and important question is how to determine whether a given abstract $3$-Lie algebra is a canonical Nambu $3$-Lie algebra? A similar question arises for abstract $n$-Lie algebras.  For $n\ge 3,$ Alexeevsky and Guha \cite{AG} and Marmo-Vilasi-Vinogradov \cite{MVV} showed that the $n$-Lie algebra structure on a given Nambu-Poisson manifold can be realized locally by the canonical Nambu bracket. In other words, there exists a local chart (coordinates) $x_1, \ldots, x_n$ in a neighbourhood of a point, with certain conditions, in the manifold such that the corresponding Poisson bracket of this $n$-Lie algebra is given by $(\ref{nb})$. However, there is still a long way to go to answer the question globally and completely.

Though the question is difficult to answer, there is evidence in the literature showing that the question is related to the realization problem of $3$-Lie algebras. For instance, using the traditional generating method, Ding {\it et al.}  \cite{Di} constructed a family of $3$-Lie algebras, including the $w_\infty$ $3$-Lie algebra and {\it SDiff}$\,(T^2)$. Using the present terminology, these $3$-Lie algebras are canonical Nambu $3$-Lie algebras. Bai {\it et al.} \cite{B4} realized some $3$-Lie algebras using two-dimensional extensions of metric Lie algebras; it is not clear if the bracket of these $3$-Lie algebras can be realized by (\ref{nb}). Pozhidaev \cite{Fi, Po1, Po2, Po3} obtained some simple $3$-Lie algebras; it is unknown whether these $3$-Lie algebras are canonical Nambu $3$-Lie algebras.

In this paper, we will shed some light on this question for $3$-Lie algebras. Gathering some preliminary definitions and notation about $n$-Lie algebras in section 2, we then in section 3 construct an infinite-dimensional $3$-Lie algebra from a given commutative associative algebra $A$ using a derivation of $A$. We find that this $3$-Lie algebra is a solvable canonical Nambu $3$-Lie algebra. Its regular representation is a weight module, but not a Harish-Chandra module. Furthermore, the derivation algebra of this $3$-Lie algebra is a direct sum of a Witt algebra $W$ and an irreducible $W$-module.

In section 4, we define another infinite-dimensional $3$-Lie algebra $A_\omega^\delta$ from the same commutative associative algebra $A$ using both an involution and a different derivation of $A$. We show that this $3$-Lie algebra is a simple canonical Nambu $3$-Lie algebra. Moreover, it contains a Lie subalgebra isomorphic to a Witt algebra.

In section 5, we prove that the inner derivation algebra of $A_\omega^\delta$ is a direct sum of a Witt algebra $W$ and the smallest full toroidal Lie algebra $V$ without center. Moreover, the Lie algebra $V$ is a direct sum of three isomorphic irreducible $W$-modules. 

In section 6, we show that the regular representation of $A_\omega^\delta$ is a Harish-Chandra module, and that the natural representation of the inner derivation algebra of $A_\omega^\delta$ is an intermediate series module.

From now on, let $\mathbb F$ be a field of characteristic zero unless stated otherwise, and let $\mathbb Z$ be the set of integers.

\section{Preliminaries}
We collect some basic notation and facts about $n$-Lie algebras. A vector space $L$ over a field $\mathbb F$ is an arbitrary $n$-Lie algebra [1] if there
is an $n$-ary multilinear operation $ [\ , \,\ldots, \ ] $ satisfying
the following anticommutative identity (\ref{ac}) and the generalized Jacobi identity (\ref{jacobi})
\begin{eqnarray}
  [u_1, \cdots, u_n] &=& (-1)^{\epsilon(\sigma)}[u_{\sigma (1)}, \cdots, u_{\sigma(n)}], \label{ac}\\
~[[u_1, ~\cdots, ~u_n], ~v_2, ~\cdots, v_n ] &=& \sum_{i=1}^n[u_1, ~\cdots, ~[u_i, ~v_2, ~\cdots, ~v_n], ~\cdots, ~u_n], \label{jacobi}
\end{eqnarray}
where $u_i, v_i\in L$ and $\sigma$ runs over the symmetric group, and the number $\epsilon(\sigma)$ is the parity of the permutation $\sigma.$ It is clear that $2$-Lie algebras are the regular Lie algebras.

Let $L$ be an arbitrary $n$-Lie algebra. A derivation of $L$ is a linear map
$\delta: L\rightarrow L$, such that
$$
 \delta([u_1, ~\cdots, ~u_n])=\sum_{i=1}^n [u_1, ~\cdots, ~\delta(u_i), ~\cdots, ~u_n], \quad \text{ for all } u_1, \cdots, u_n \in L.
$$
The set Der$L$ of all  derivations of $L$ is a Lie subalgebra of the general Lie algebra gl$(L)$. This subalgebra is called the derivation algebra of $L$.

The map ad$(u_1, ~\cdots, ~u_{n-1}):$ $ ~L\rightarrow ~L$ defined
by
\[
    {\rm ad}(u_1, ~\cdots, ~u_{n-1})(u_n)=[u_1, ~\cdots, ~u_{n}],\quad \text{ for all } u_1, \cdots u_n \in L.
\]
is called a left multiplication of $L$. Clearly, each left multiplication is a derivation. The set ad$L$ of all finite linear combinations of left multiplications is an ideal of Der$L$. Each element of ad$(L)$ is referred to as an inner derivation of $L$. If $u_1, \cdots, u_{n-1}, v_1, \cdots, v_{n-1} \in L$, then
\[
    [\mbox{ad}(u_1, \cdots, u_{n-1}), ~\mbox{ad}(v_1, \cdots, v_{n-1})]=\sum\limits_{i=1}^{n-1}\mbox{ad}(v_1, \cdots, [u_1, \cdots,
u_{n-1}, v_i], \cdots, v_{n-1}).
\]
If $L_{1}, L_{2},\cdots, L_{n}$ are subalgebras of $L$, let $[L_1, $ $L_{2},$  $\cdots, $ $L_{n}]$ denote the
subspace of $L$ generated by all vectors $[u_{1},$ ~$ u_{2}, $
~$\cdots, ~u_{n}]$, where $u_{i}\in L_{i}$ for $i=1, ~\cdots,
~n$.
 Let $L^1=[L, ~L, ~\cdots, ~L]$ with $n$ copies of $L$. Then $L^1$ is a subalgebra of $L$, called the derived algebra of $L$. We say that $L$ is abelian if its derived algebra is zero.

An ideal of $L$ is a subspace $I$ such that $ [I,
L, \cdots, L]\subseteq I.$  We say that $L$ is simple if $L$ is not abelian and it has no ideals other than $0$ and itself.
An ideal $I$ of $L$ is a solvable ideal, if $I^{(r)}=0$ for some $r\geq 0$, where $I^{(0)}=I$ and $I^{(s)}$ is defined by induction,
$$
I^{(s+1)}=[I^{(s)},  I^{(s)}, L,\,\cdots, L]
$$
for $0\le s\le r$. If $L$ is a solvable ideal of itself, then $L$ is called a solvable $n$-Lie algebra. An ideal $I$ of $L$ is called nilpotent, if
$I^{r}=0$ for some $r\geq 0$, where $I^{0}= I$ and $I^s$ is defined by
\[
    I^{s+1}=[I^{s}, I,  L, ~\cdots, ~L]
\]
for $0\le s\le r$. If $L$ is a nilpotent ideal of itself, we call $L$ a nilpotent $n$-Lie algebra. An ideal $I$ of $L$ is referred to as hypo-nilpotent, if $I$ is a nilpotent $n$-Lie algebra, but it is not a nilpotent ideal of $L$.

 A Cartan subalgebra of $L$ is a nilpotent subalgebra of $L$ such that for all $y\in L$, if $[y, H, L,\cdots, L]\subseteq H$, then $y\in H$.

\begin{lemma}{\rm(\cite{B4})}\label{vf}
  Let $(L, [ \,, ])$ be a Lie algebra over a field $\mathbb F$, and let $L^*$ be the dual space of $L$. Suppose $f\in L^*$ satisfies $f([u, v])=0$, for all $u, v\in L$. Then $(L, [\, , , ]_f)$ is a $3$-Lie algebra with respect to the multiplication
\begin{equation*}
[u, v, w]_{f}=f(u)[v, w]+f(v)[w, u]+f(w)[u, v], \quad \text{ for all } u, v, w\in L.
 \end{equation*}
\end{lemma}

\begin{lemma}{\rm(\cite{B5})}\label{do}
Let $\delta$ be a derivation and $\omega$ be an involution of a commutative associative algebra $A$ over $\mathbb F$. If
$\omega\delta+\delta \omega=0,$ then $A$ is a $3$-Lie algebra under the
multiplication
\begin{equation}\label{w3}
    [u, v, w]=\det  \begin{pmatrix}
                        \omega(u) & \omega(v) &\omega(w) \\
                        u & v & w  \\
                        \delta(u) & \delta(v) & \delta(w) \\
                    \end{pmatrix}, \quad \text{for all }  u, v, w\in A.
\end{equation}
\end{lemma}

To describe the connection of the structures of the $3$-Lie algebras in this paper to the representations of Witt algebras, we need some basic facts about the Witt algebras. A basis for a Witt algebra $W$ is given by the vector fields $L_n=-z^{n+1} \frac{\partial}{\partial z}$, for $n$ in $\mathbb Z$. The Lie bracket of two vector fields is given by
\[
    [L_i,L_j]=(i-j)L_{i+j}.
\]
This algebra has a central extension called the Virasoro algebra that is important in conformal field theory and string theory. A module $V$ over the Witt algebra $W$ is referred to as a weight module if it admits a weight space decomposition $V=\oplus_{V\in\mathbb F} V_\lambda$ where
\[
    V_\lambda = \{v\in V\mid L_0v = \lambda v\}.
\]
A weight module over $W$ is an intermediate series module if the dimension of every weight space is less than or equal to one.

Let $L$ be a $3$-Lie algebra and $V$ be a vector space over $\mathbb F$.
An $L$-module is a pair $(V, \rho)$ where $\rho$ is a linear mapping
$
    \rho: L^{\wedge 2}\rightarrow \mbox{End}(V)
$
satisfying
\begin{align*}
    \rho([x_1, x_2, x_3],\, y)      &= \rho(x_3,\,x_1)\rho(x_2,\, y)+ \rho(x_2,\,x_3)\rho(x_1,\, y) + \rho(x_1,\, x_2)\rho(x_3,\, y) \\
    [\rho(x_1,  x_2),\, \rho(y_1, y)] &= \rho(x_1,\, x_2)\rho(y_1,\,y) - \rho(y_1,\, y) \rho(x_1,\, x_2)\\
                                        &=\rho(y_1,\, [x_1,\, x_2,\, y]+\rho([x_1,\, x_2,\, y_1],\, y)
\end{align*}
for all $x_i, y_1, y\in L, \, i=1, 2, 3$.

A module $(V, \rho)$ of a $3$-Lie algebra $L$ is referred to as a {\it weight module} if it admits a weight space decomposition $V=\oplus_{\lambda\in\mathbb (H\wedge H)^*} V_\lambda$, where $H$ is a Cartan subalgebra of $L$, $\lambda: H\wedge H\rightarrow \mathbb F$ is a linear mapping, and
\[
    V_\lambda = \{v\in V\mid \rho(h_1, h_2)v = \lambda(h_1, h_2) v, \quad \text{for all } h_1, h_2\in H\}.
\]
A module of $L$ is a {\it Harish-Chandra} module if it is a weight module and every weight space is finite-dimensional. A Harish-Chandra module of $L$ is an {\it intermediate series module} if the dimension of every weight space is less than or equal to one.

\section{The $3$-Lie algebra $(A, [ , , ]_{f,\, k})$}\label{fk}
Let $A$ be a commutative associative algebra with a basis
$\{L_r, M_r \mid  r \in \mathbb Z \}$ and the multiplication
\begin{equation}\label{a}
    L_rL_s=L_{r+s},\quad M_rM_s=M_{r+s}, \quad L_rM_s=0, \quad\text{for all } r, s\in \mathbb Z.
\end{equation}
For each $k\in\mathbb Z$, we define a derivation $d_k$ of $A$ by
\begin{eqnarray*}
  d_k (aL_r+bM_t) = arL_{k+r},
\end{eqnarray*}
for all $a, b\in\mathbb F, ~r, t\in \mathbb Z$. Through $d_k$, each integer $k$ induces a Lie algebra structure $[\, , ]_k$ on $A$ defined by
\begin{eqnarray}\label{dms}
    [u, v]_k = d_k(u) v - u d_k(v), \quad \text{for all } u, v\in A.
\end{eqnarray}
We will show that each integer $k$ gives rise to a canonical Nambu $3$-Lie algebra structure on $A$.

To simplify notation and expressions, from now on, we assume that any bracket not listed in the multiplication of a $3$-Lie algebra is always zero.
\begin{theorem}\label{dtf}
  Let $f: A \rightarrow\mathbb F$ be a nonzero linear function with $f(L_r)=0$ for all $r\in \mathbb Z$. Then for any integer $k$, $A$ is a canonical Nambu $3$-Lie algebra in the multiplication $[\, , , ]_{f,\,k}$ defined by
\begin{equation} \label{dmm}
    [L_r, L_s, M_t]_{f,\,k}= f(M_t)(r-s) L_{r+s+k}, \quad\text{for all } L_r, L_s, M_t\in A.
\end{equation}
Furthermore, $A$ is a non-solvable $3$-Lie algebra.
\end{theorem}

\begin{proof} From the definition of $f$ and the Lie structure (\ref{dms}) of $A$ it follows that $f([A, A])=0$. By Lemma \ref{vf} we know that $A$ is a $3$-Lie algebra under the multiplication (\ref{dmm}).

 Now we show that this $3$-Lie algebra is a canonical Nambu $3$-Lie algebra. It suffices to show that it can be realized as a canonical Nambu $3$-Lie algebra. Let $\beta_r=f(M_r)$, and choose the basis elements $L_r$ and $M_{r}$ of $A$ to be
\[
    L_r = z\exp (rx), \quad M_r= \beta_r y\exp (kx), \quad\text{for all } r \in \mathbb Z,
\]
where $\exp$ is the exponential function. Applying the Nambu bracket $[\, , , ]_\partial$ defined in (\ref{nb}) to these basis elements, we obtain
\[
    [L_r, M_s, M_t]_\partial=[M_r, M_s, M_t]_\partial=[L_r, L_s,L_t]_\partial=0,
\]
and
\[
    [L_r, L_s, M_t]_\partial=\begin{vmatrix}
                            rz \exp(rx) & 0 & \exp(rx) \\
                            sz \exp(sx) & 0 & \exp(sx)\\
                            \beta_t y\exp(kx) & \beta_t\exp(kx)&0 \\
    \end{vmatrix}=\beta_t(r-s)L_{r+s+k}.
\]
So, the Nambu bracket coincides with the bracket $[\, , , ]_{f,\,k}$ on $A$. In other words, the $3$-Lie algebra $A$ with the bracket $[\, , , ]_{f,\,k}$ can be realized as the following canonical Nambu $3$-Lie algebra
\[
    A=\sum\limits_{r\in \mathbb Z}(\mathbb F z\exp (rx)) \oplus \sum\limits_{r\in\mathbb Z}(\mathbb F y\exp (rx)).
\]

Next, a direct computation yields that $A^1=\sum\limits_{r\in \mathbb Z} \mathbb F L_r$, and $[A^1, A^1, A^1]_{f,\, k}=0$, but
$A^{(s+1)}=[A^{(s)}, A^{(s)}, A]_{f, k}=A^1\neq 0$ for all $s\geq 1$. It follows that $A$ is a non-solvable $3$-Lie algebra. \hfill$\Box$

\end{proof}

\begin{corollary}
  The subalgebra $L=\sum\limits_{s\in \mathbb Z} \mathbb F L_s $ is a hypo-nilpotent minimal ideal of $A$, and $M=\sum\limits_{s\in \mathbb Z} \mathbb F M_s$ is an abelian subalgebra.
\end{corollary}

\begin{corollary}
    The regular representation of the $3$-Lie algebra $A$ is a weight module, but not a Harish-Chandra module.
    \end{corollary}
\begin{proof}
Let $L=A$ be the $3$-Lie algebra $(A,\, [\,,\,,]_{f, k})$. Then the Cartan subalgebra of $L$ is
$$
    H = \mathbb F L_{-k}\oplus \sum\limits_{t\in \mathbb Z}\mathbb F M_t,
$$
which is an abelian subalgebra. The regular representation of $L$ is $(V, \mbox{ad})$, where $V=A$ as a vector space, and
$\mbox{ad}: L\wedge L\rightarrow \mbox{End}(V)$ is given by
\[
    {\rm \mbox{ad}}(x, y)(z)=[x, y, z]_{f, k}\quad \text{for all } x, y\in L  \text{ and } z\in V.
\]
Define $\lambda_r: H\wedge H\rightarrow \mathbb F$ by
\[
    \lambda_r(L_{-k}, M_t)=f(M_t)(r+k),\quad \lambda_r(M_s, M_t)=0, \quad \text{for all } r\in \mathbb Z.
\]
Then $V=\sum\limits_{\lambda_r} V_{\lambda_r} $ where $V_{\lambda_r}=\mathbb F L_r$ if $r\neq 0$, and $V_{0}=H$. But, $\dim H$ is not finite. \hfill$\Box$
\end{proof}

We now describe the inner derivation algebra ad$A$ of $(A, [ , , ]_{f,\, k})$. We begin by introducing three types of left multiplications of $(A, [ , , ]_{f,\,k})$,
\[
    W_{r,s}=\mbox{ad}(L_r, M_s), \quad X_{r,s}=\mbox{ad}(L_r, L_s), \quad Y_{r,s}=\mbox{ad}(M_r, M_s), \quad\text{for all }r, s\in \mathbb Z.
\]
Define two subspaces $W$ and $X$ of ad$A$ by
\[
    W =\sum\limits_{r\in \mathbb Z}\mathbb FW_{r,\, 0}\qquad\text{and}\qquad X = \mathbb FX_{1,\, -1} + \sum\limits_{r\in \mathbb Z,\, r\neq 0}\mathbb FX_{r,\, 0}.
\]
\begin{theorem}
The Lie subalgebra $W$ of {\rm ad}$A$ is isomorphic to a Witt algebra, and $X$ is an irreducible $W$-module. Furthermore,
\[
    {\rm ad}A = W \oplus X
\]
is not solvable, and $X$ is the only minimal ideal of ${\rm ad}A$.
\end{theorem}

\begin{proof} Let $\beta_s=f(M_s)$. Using (\ref{dmm}), it is routine to check that $Y_{r,\,s}= 0$ and, for all $r,\, s,\,  t\in\mathbb Z$,
\begin{align*}\label{wx}
    W_{r,\, s}(L_t) &= \beta_s (t-r)L_{r+t+k}, \\
    W_{r,\, s}(M_t) &= X_{r,\, s}(L_t)=0,\\
    X_{r,\, s}(M_t) &= \beta_t (r-s)L_{r+s+k},\\
    Y_{r,\,s}(M_t) &=Y_{r,\,s}(L_t)=0.
\end{align*}
Since $f\neq 0,$ there exists $s_0\in \mathbb Z$ such that $\beta_{s_0}\neq 0$. A direct computation shows that if $s\neq -r,$ then
\[
    X_{r,\, s} = \frac{r-s}{r+s}X_{r+s,\,  0}.
\]
If $ s=-r$, then $X_{r,\, s}=\frac{1}{r}X_{1,\,  -1}$.
Hence $W_{r,\, s}=\frac{\beta_s}{\beta_0}W_{r,\, s_0}$ and  $X_{r,\, s} = -X_{s,\, r}$, for $r,\, s\in \mathbb Z.$ It follows that
\begin{equation*}\label{b}
    \{W_{s,\,  s_0}, ~X_{r,\, 0}, ~X_{1,\,  -1} \mid ~ r,\, s\in \mathbb Z, ~r\neq 0\}
\end{equation*}
is a basis of the inner derivation algebra ad$(A)$.
With (\ref{dmm}), we obtain
\[
    [W_{r,\, s_0}, ~W_{s,\,  s_0}]=\beta_{s_0}(s-r)W_{s+r+k,\,  s_0}.
\]
Thus $W=\sum\limits_{s\in \mathbb Z}\mathbb F W_{s,\,  s_0}$ is isomorphic to a Witt algebra. Thus ad$A$ is not solvable.

Without showing the details, a moderate computation gives rise to $[X_{r,\, 0},~~ X_{s,\,  0}]=[X_{r,\,  0},\, X_{1,\,  -1}]=0$, and
\begin{equation*}\label{wsx1}
\begin{array}{l}
\begin{array}{l}
{[}W_{s,\,  s_0}, ~X_{1,\, -1}]=\begin{array}{l}
\left\{\begin{array}{l}
2\beta_{s_0}\frac{t-s}{s+k}X_{s+k,\, 0},\, ~~~~~~~~~\quad s\neq -k, \\\\
2\beta_{s_0}\frac{1}{s}X_{-1,\,  1},\,~~~~~~~~~~~~~~~~ s=-k,\\
\end{array}\right.  \end{array}\\\\
\end{array} \end{array}
\end{equation*}
as well as
\begin{equation*}\label{wsxr}
\begin{array}{l}
\begin{array}{l}
{[}W_{s,\,  s_0}, ~X_{r,\, 0}]=\begin{array}{l}
\left\{\begin{array}{l}
r\beta_{s_0}\frac{r-s+k}{r+s+k}X_{r+s+k,\, 0}\,, \,~\quad r+s\neq -k, \\\\
r\beta_{s_0}\frac{1}{s}X_{-1,\,  1}\,, ~~~~~~~~~~~~\quad  r+s=-k, ~s\neq 0,\\\\
r\beta_{s_0}X_{r+k,\,  0}\,, ~~~~~~~~~~\quad\quad  r=-k, ~ s=0.
\end{array}\right.  \end{array}\\\\
\end{array}
\end{array}
\end{equation*}
It follows that $X=FX_{1,\,  -1}+\sum\limits_{r\in \mathbb Z,\,  r\neq 0} FX_{r,\,  0}$ is an irreducible $W$-module. Moreover, ad$A = W \oplus X$, and $X$ is its only minimal ideal. \hfill $\Box$

\end{proof}

\section{The $3$-Lie algebra $A_{\omega}^\delta$}\label{awd}
Again, let $A$ be the commutative associative algebra defined by $(\ref{a})$. Define the derivation $\delta$ of $A$ by
\begin{equation}\label{delta}
    \delta(L_r)=r L_r, \quad \delta(M_r)= r M_r, \quad\text{for all } r\in \mathbb Z.
\end{equation}
Let $\omega : A \rightarrow A$ be a linear mapping given by
\begin{equation}\omega(L_r)=M_{-r}, \quad \omega(M_r)=L_{-r}, \quad\text{for all } r\in \mathbb Z. \end{equation}  Then  $\omega$ is an involution of $A$, since $\omega^2=\omega$ and
$$\begin{array}{l}
\begin{array}{l}
\omega(L_rL_s)=\omega(L_{r+s})=M_{-r-s}=M_{-r}M_{-s}=\omega(L_r)\omega(L_s), \text{ and }\\
\omega(M_rM_s)=\omega(M_{r+s})=L_{-r-s}=L_{-r}L_{-s}=\omega(M_r)\omega(M_s).\\
\end{array}  \end{array}
$$
Furthermore,
\begin{equation}\label{omega}
    (\delta \omega+\omega \delta)(L_r)=(\delta \omega+\omega \delta)(M_r)=0, \quad\text{for all } r\in \mathbb Z.
\end{equation}

We have the following result.
\begin{theorem}
$A$ is a simple canonical Nambu $3$-Lie algebra under the multiplication $[\, , , ]_{\omega}$ defined by, for all $r, s, t\in \mathbb Z$,
\begin{align}
           [L_r, L_s, M_t]_{\omega}&=(s-r)L_{r+s-t},\label{llm}\\
           [L_r, M_s, M_t]_{\omega}&=(t-s)M_{s+t-r}.\label{lmm}
\end{align}
\end{theorem}

\begin{proof}
Since the derivation $\delta$ and involution $\omega$ satisfy (\ref{delta}) and (\ref{omega}), by Lemma 2.2 we know that $A$ is a $3$-Lie algebra with respect to the multiplication (\ref{w3}), which leads to  (\ref{llm}) and (\ref{lmm}), and the remaining brackets are all zero.  Let $A_{\omega}^{\delta}$ denote this $3$-Lie algebra $(A, [ , , ]_{\omega})$.

Now we show that $A_{\omega}^{\delta}$ is a canonical Nambu $3$-Lie algebra. It suffices to show that it can be realized as a canonical Nambu $3$-Lie algebra. Choose the basis elements $L_r$ and $M_{r}$ of $A$ to be
$$
    L_r=z \exp(rx), ~~ M_r=y \exp(-rx), ~ \mbox{for all } ~ r\in \mathbb Z,
$$
Then the Nambu bracket on $A$ becomes  \begin{align*}
[L_r, L_s, M_t]_\partial&=\frac{\partial(L_r, L_s, M_t)}{\partial(x, y, z)}
=\begin{vmatrix}
rz \exp(rx) & 0 & \exp(rx) \\
sz \exp(sx) & 0 & \exp(sx)\\
-ty\exp(-tx) & \exp(-tx)&0 \\
\end{vmatrix}=(s-r)L_{r+s-t},\\
[L_r, M_s, M_t]_\partial&=\frac{\partial(L_r, M_s, M_t)}{\partial(x, y, z)}=\begin{vmatrix}
rz \exp(rx) & 0 & \exp(rx) \\
-sy\exp(-sx) & \exp(-sx) & 0\\
-ty\exp(-tx)&\exp(-tx)&0\\
\end{vmatrix}=(t-s)M_{s+t-r},\\
[M_r, M_s, M_t]_\partial&=\frac{\partial(M_r, M_s, M_t)}{\partial(x, y, z)}=\begin{vmatrix}
-ry \exp(rx) & \exp(rx)& 0 \\
-sy\exp(-sx) & \exp(-sx) & 0\\
-ty\exp(-tx)&\exp(-tx)&0
\end{vmatrix}=0,\\
[L_r, L_s,L_t]_\partial&=\frac{\partial(L_r, L_s, L_t)}{\partial(x, y, z)}=\begin{vmatrix}
rz \exp(rx) & 0 & \exp(rx) \\
sz \exp(sx) & 0 & \exp(sx)\\
tz \exp(tx) & 0 & \exp(tx)
\end{vmatrix}=0.
\end{align*}
Thus the bracket $[\, , , ]_{\omega}$ and the Nambu bracket on $A$ are consistent, showing that the $3$-Lie algebra $A$ with the bracket $[\, , , ]_{\omega}$ is a canonical Nambu $3$-Lie algebra.

We next show that $A_{\omega}^{\delta}$ is simple. For every $r, s\in \mathbb Z$ and $r\neq s$, from (\ref{llm}) and (\ref{lmm}) we obtain that both
\[
    [L_r, L_s, M_s]_{\omega}=(s-r)L_{r} \quad\text{and}\quad [L_s, M_r, M_s]_{\omega}=(s-r)M_{r}
\]
are in the derived algebra of $A^{\delta}_{\omega}$. So, $A^{\delta}_{\omega}$ is equal to its derived algebra. Now, let $J$ be a non-zero ideal of $A^{\delta}_{\omega}$. If $0\ne u\in J,$ then there exist $a_i, b_j\in\mathbb F$ such that
\[
    u=\sum\limits_{i=p}^qa_iL_i+\sum\limits_{j=m}^nb_jM_j, \quad
p < q, ~m < n, ~a_p\neq 0 ~or ~b_m\neq 0.
\]
Let $r\in\mathbb Z$ such that $r > q$ and $r > n.$ Then from (\ref{llm}) and (\ref{lmm}) we have
$$
\mbox{ad}^l(L_r, M_r)(u) = \sum\limits_{i=p}^qa_i(i-r)^lL_i + \sum\limits_{j=m}^n b_j(j-r)^lM_j \in J, \quad \text{for any positive integer } l.
$$
By Vandermonde determinant, we obtain $L_p\in J$ or $M_m\in J$. Therefore, for every $s\in Z, ~s\neq p,$
$$
    L_s = \frac{1}{p-s}[L_s, L_p, M_p]_{\omega}\in J \quad\text{and}\quad M_s = \frac{1}{p-s}[L_p, M_s, M_p]_{\omega}\in J.
$$
From (\ref{lmm}), $M_p=\frac{1}{s-p}[L_s, M_p, M_s]_{\omega}\in J.$ Hence $J=A$, showing that $A$ is simple. \hfill$\Box$
\end{proof}

The $3$-Lie algebra structure $[\, ,  ,]_\omega$ of $A$ induces different Lie algebra structures on $A$. For a fixed $k\in\mathbb Z$, define the Lie bracket $[\, ,]_{L_k}$ on $A$ by
\[
    [u, v]_{L_k}=[u, v, L_k]_\omega,  \quad\text{for all } u, v\in A.
\]
\begin{corollary}\label{awl}
  Let $M=\sum\limits_{s\in\mathbb Z}\mathbb F M_s$ and $L=\sum\limits_{s\in \mathbb Z}\mathbb F L_s$. Then $M$ is a Lie subalgebra of $(A, ~[\, ,]_{L_k})$  isomorphic to a Witt algebra, and $L$ is a reducible $M$-module. Furthermore, $\mathbb F L_k$ is the center of $(A, ~[\, , ]_{L_k})$.
\end{corollary}
\begin{proof}
  The results follow from  $[L_r, L_s]_{L_k} =0,$ as well as
  \[
    [M_r, M_s]_{L_k} = (s-r)M_{r+s-k} \quad\text{and}\quad
    [L_r, M_s]_{L_k} =(r-k)L_{r+k-s}
  \]
   for all $r, s\in \mathbb Z.$ \hfill$\Box$
\end{proof}

Similarly, for each $k\in\mathbb Z$, we can define a new Lie bracket on $A$ by
\[
    [u, v]_{M_k}=[u, v, M_k]_\omega,  \quad\text{for all } u, v\in A,
\]
a similar result to Corollary \ref{awl} holds.

\section{The inner derivation algebra of $A^{\delta}_{\omega}$}\label{Awd}
To describe the structure of ${\rm ad} A^{\delta}_\omega$ and its connection to the representations of Witt algebras, we need a basis of ad$A_\omega$. For $0\ne r\in\mathbb Z$, let
\begin{align}
p_r &= \frac{1}{2}(\mbox{ad}(L_0, M_{-r})+\mbox{ad}(L_r, M_0)) & q_r &=\frac{1}{r}(\mbox{ad}(L_0, M_{-r})-\mbox{ad}(L_r, M_0)),\label{nzeros}\\
x_r &= \frac{1}{r}\mbox{ad}(L_r, L_0)                          & z_r &=\frac{1}{-r}\mbox{ad}(M_{-r}, M_0), \label{nzeros2}
\end{align}
and let
\begin{align}
    p_0 &= \mbox{ad}(L_0, M_{0})                & q_0 &= \mbox{ad}(L_0, M_{0})-\mbox{ad}(L_1, M_1), \label{zeros}\\
    x_0 &=\frac{1}{2}\mbox{ad}(L_{1},L_{-1})    & z_0 &= \frac{1}{2}\mbox{ad}(M_{1},M_{-1}). \label{zeros2}
\end{align}

\begin{theorem}\label{newbasis}
The set
\begin{equation*}
    \{p_r,\, q_r,\, x_r,\, z_r\mid r\in\mathbb Z\}
\end{equation*}
is a basis of ${\rm ad}A^{\delta}_\omega$, and the multiplication in this basis is as follows: for all $r, s\in \mathbb Z,$
\begin{subequations}
\begin{align}
    [p_r, p_s] &= (r-s)p_{r+s}, &   [p_r, q_s]    &=-sq_{r+s}, \label{sb1}  \\
    [p_r, x_s] &=-sx_{r+s},     &   [p_r, z_s]  &=-sz_{r+s}, \label{sb2} \\
    [q_r, q_s] &=0,             &   [q_r,x_s]     &=-2x_{r+s}, \label{sb3} \\
    [q_r, z_s] &=2z_{r+s},      &   [x_r, x_s]    &=0, \label{sb4} \\
    [z_r, z_s] &=0,             &   [z_r, x_s]    &=q_{r+s}.\label{sb5}
\end{align}
\end{subequations}
\end{theorem}
\begin{proof} For convenience, write
\begin{equation}\label{xyz}
    W_{r, s}=\mbox{ad} (L_r, M_s), \quad X_{r, \,s}=\mbox{ad} (L_r, L_s), \quad  Y_{r, \,s}=\mbox{ad} (M_r, M_s)
\end{equation}
where  $r, s\in \mathbb Z.$
Then, for $t\in\mathbb Z$,
\begin{subequations}
\begin{align}
  W_{r, \,s}(L_t) &= (r-t) L_{t+r-s} ~\quad\quad\quad\text{and}         & W_{r, \,s}(M_t) &= (t-s)M_{t+s-r} \label{WLM}\\
  X_{r, \,s}(L_t) &= 0  ~~\qquad\qquad\quad\quad\quad\quad\text{and}    & X_{r, \,s}(M_t) &= (s-r) L_{r+s-t} \label{XLM}\\
  Y_{r, \,s}(L_t) &= (s-r)M_{r+s-t}  \quad\quad\quad\text{and}          & Y_{r, \,s}(M_t) &= 0. \label{YLM}
\end{align}
\end{subequations}
A tedious calculation yields that the set
\begin{equation}\label{fb}
    \{W_{0,\, 0},\, W_{1,\, 1},\, X_{1,\, -1},\,  Y_{1,\, -1} \} \bigcup \{W_{r,\, 0},\, W_{0,\, r},\, X_{r,\, 0},\, Y_{r,\, 0}~ \mid r\in \mathbb Z, \,r\neq 0\}
\end{equation}
is linearly independent in ad$A^{\delta}_{\omega}$.

Next, we show that (\ref{fb}) is a basis of ad$A^{\delta}_{\omega}$.
By (\ref{WLM}), if $s=r$, we have  $W_{r,\,s} = (1-r)W_{0,\, 0} + rW_{1,\,1}$; if $s\neq r$, then
\begin{equation*}\label{yrs}
    W_{r,\, \,s} = \frac{1}{r-s}(rW_{r-s,\, 0} - sW_{0,\, s-r}).
\end{equation*}
From (\ref{XLM}) and (\ref{YLM}),
if $s=-r$, then $X_{r,\, s}=rX_{1,\, -1}$ and $Y_{r,\, s}=rY_{1,\, -1}$; if $ s\neq -r$, then we obtain
\begin{equation*}\label{xrs}
  X_{r,\, s} = \frac{r-s}{r+s}X_{r+s,\, 0}\quad\text{and}\quad  Y_{r,\, s} = \frac{r-s}{r+s}Y_{r+s,\, 0}.
\end{equation*}
Also, $X_{r,\, \,s} = -X_{s,\, r}$ and $Y_{r,\, s} = -Y_{s,\, r},$ for all $r,\, s\in \mathbb Z.$
Thus, the elements in (\ref{fb}) form a basis of ad$A^{\delta}_{\omega}$.

From the definitions of $p_r,\, q_r,\, x_r,\, z_r$ given in (\ref{nzeros})$-$(\ref{zeros2}), we know that
\begin{align*}
    p_0 &=W_{0,\, 0}                           & q_0 &= W_{0,\, 0}-W_{1,\, 1}\\
    x_0 &=\frac{1}{2}X_{1,\,  -1}              & z_0 &=\frac{1}{2}Y_{1,\, -1}
\end{align*}
and, for $0\ne r\in \mathbb Z$,\,
\begin{align*}
    p_r &= \frac{1}{2}(W_{r,\, 0}+W_{0,\,  -r})       & q_r &= \frac{1}{r}(W_{0,\,  -r}-W_{r,\,  0}) \\
    x_r &= \frac{1}{r}X_{r,\,  0}                  & z_r &= \frac{1}{-r}Y_{-r,\,  0}.
\end{align*}
Therefore, $\{p_r,\, q_r,\, x_r,\, z_r\}$  is a basis of ad$A^{\delta}_\omega$.

To obtain the multiplication in (17a)$-$(17e), we have to calculate the Lie brackets of any two elements in (\ref{xyz}). Note that if a bracket is zero, we omit it.

From (\ref{WLM}) it follows that, for all $r, s\in\mathbb Z$,
\begin{align*}
  [W_{r,\, 0}, ~~ W_{s,\, 0}] &= (r-s)W_{r+s,\, 0}  & [W_{0,\, r}, ~~W_{0,\, s}]  &= (s-r)W_{0,\, r+s} \\
  [W_{1,\,1},  ~~ W_{r,\, 0}] &= -rW_{r,\, 0}       & [W_{1,\,1}, ~~  W_{0,\, s}] &= sW_{0,\, s}       \\
  [W_{0,\,0},  ~~  W_{r,\, 0}] &= -rW_{r,\,0}       & [W_{0,\,0}, ~~  W_{0,\, s}] &= ~ sW_{0,\, s}
\end{align*}
and
\[
    [W_{r,\,0}, ~~ ~~ W_{0,\, s}] ~=\left\{\begin{array}{l}
 \frac{1}{r-s}(r^2W_{r-s,\, 0}-s^2W_{0,\,s-r}), ~~  r\neq s,\, \\\\
 2rW_{0,\, 0}+r^2(W_{1,\, 1}-W_{0,\, 0}),~~ r=s.\\
\end{array}\right. \\
\]

Using (\ref{WLM}) and (\ref{XLM}), we obtain
\begin{align*}
    [X_{r,\, 0}, ~~ W_{0,\, s}] &=rX_{r-s,\, 0}     & [X_{1,\,-1}, ~~ W_{0,\, s}] &=2X_{-s,\, 0}\\
    [X_{r,\, 0}, ~~ W_{0,\, 0}] &=rX_{r,\, 0}       & [X_{1,\,-1}, ~~ W_{r,\, 0}] &=-2X_{r,\, 0}\\
    [X_{r,\, 0}, ~~ W_{1,\,1}]  &=(r-2)X_{r,\, 0}   & [X_{1,\,-1}, ~~ W_{1,\, 1}] &=-2X_{1,\, -1}
\end{align*}
and
\[
    [X_{r,\, 0}, ~~ W_{s,\,0}] = \begin{cases} r^2 X_{1,\, -1}, ~\qquad s=-r,\\
                                               \frac{r(r-s)}{r+s}X_{r+s,\, 0}, ~\, s\neq -r.
                                               \end{cases}
\]
It follows from (\ref{WLM}) and (\ref{YLM}) that
\begin{align*}
    [Y_{r,\, 0}, ~~ W_{s,\,0}]  &=-rY_{r-s,\, 0}        & [Y_{1,\, -1}, ~~ W_{0,\,s}] &=2Y_{s,\,0}\\
    [Y_{r,\, 0}, ~~ W_{0,\,0}]  &= -rY_{r,\,0}          & [Y_{1,\, -1}, ~~ W_{r,\, 0}]&=-2Y_{-r,\, 0}\\
    [Y_{r,\, 0}, ~~ W_{1,\,1}]  &=(2-r)Y_{r,\, 0}       & [Y_{1,\, -1}, ~~ W_{1,\,1}] &=2Y_{1,\,-1}\\
\end{align*}
and
\[
    [Y_{r,\, 0}, ~~ W_{0,\,s}] = \begin{cases}
                                        -r^2Y_{1,\, -1}, \,\qquad s=-r,\\
                                        \frac{r(s-r)}{r+s}Y_{r+s,\, 0}, \quad s\neq -r.
                                        \end{cases}
\]
The following are obtained from (\ref{XLM}) and (\ref{YLM})
\begin{align*}
[X_{r,\, 0}, ~~ Y_{1,\, -1}]  &=2(W_{r,\,0}-W_{0,\, -r})\\
[X_{1,\, -1}, ~~ Y_{s,\, 0}]  &=-2(W_{-s,\, 0}-W_{0,\, s})\\
[X_{1,\, -1}, Y_{1,\, -1}]    &=-4(W_{0,\,0}-W_{1,\, 1})\\
[X_{r,\, 0}, ~~  Y_{s,\, 0}]  &=\begin{cases}\frac{rs}{r-s}(W_{r-s,\, 0}-W_{0,\,s-r}), ~\quad r\neq s,\\
                                                r^2(W_{1,\, 1}-W_{0,\, 0}),\qquad\quad\quad  r=s,
                                                \end{cases}
\end{align*}
So, the multiplication of the basis elements of $\{p_r,\, q_r,\, x_r,\, z_r\mid r\in\mathbb Z\}$ satisfies (17a)$-$(17e).
\hfill$\Box$
\end{proof}

We are now in a position to study the structure of $ad A^{\delta}_\omega$ and its connection to the representations of Witt algebras.
Define four subspaces of ad$A^{\delta}_\omega$ by
\[
    W=\sum\limits_{r\in \mathbb Z}\mathbb Fp_r, \quad  V_1=\sum\limits_{r\in \mathbb Z}\mathbb Fq_r,
    \quad V_2=\sum\limits_{r\in \mathbb Z}\mathbb Fx_r, \quad V_3=\sum\limits_{r\in \mathbb Z}\mathbb Fz_r.
\]

\begin{theorem}
If ch $\mathbb F\neq 2$, then

{\rm (a)} $W$ is a simple Lie algebra isomorphic to a Witt algebra, and $V=V_1 \oplus V_2 \oplus V_3$ is a Lie algebra isomorphic to the simple Lie algebra $sl(2, \mathbb F)\otimes \mathbb F[t, t^{-1}]$.

{\rm (b)} The inner derivation algebra {\rm ad}$ A^{\delta}_{\omega}$ is an indecomposable Lie algebra, and it can
be decomposed into the semi-direct sum of subalgebras $W$ and $V$. Moreover, $V$ is the only minimal ideal of ${\rm ad}A^{\delta}_\omega$.

{\rm (c)} For $i=1, 2, 3$, define $\rho_i$ to be the map from $W$ to $V_i$ by
\[
    \rho_i(p_r)=\begin{cases} q_r, \quad\text{if } i=1, \\
                             x_r, \quad\text{if } i=2, \\
                              z_r, \quad\text{if } i=3,
                \end{cases}
\]
where $r\in \mathbb Z$. Then $(V_i,\,\rho_i)$ is an irreducible intermediate series $W$-module isomorphic to the regular representation of $W$.
\end{theorem}
\begin{proof}
The Lie algebra $W=\sum\limits_{r\in \mathbb Z} \mathbb F p_r $ is isomorphic to a Witt algebra, since ${[}p_r, p_s]= (r-s)p_{r+s}$ for all $r, s\in \mathbb Z$.

We next show that the Lie algebra $V$ is isomorphic to $sl(2, \mathbb F)\otimes \mathbb F[t, t^{-1}]$. Let $sl(2, \mathbb F) = \mathbb F e\oplus \mathbb Fh \otimes \mathbb Ff$.
Define the map
$
    \sigma: V\rightarrow sl(2, \mathbb F)\otimes \mathbb F[t, t^{-1}]
$
by
$$
    \sigma(q_r)=h\otimes t^r, \quad \sigma(z_r)=e\otimes t^r,  \quad \sigma(x_r)=f\otimes t^r, \quad\text{for all } r\in \mathbb Z.
$$
From Theorem \ref{newbasis} it follows that
\begin{align*}
    [\sigma(q_r), \sigma(q_s)]  &=\sigma([q_r, q_s])=0,\\
    [\sigma(x_r), \sigma(x_s)]  &=\sigma([x_s, x_s])=0,\\
    [\sigma(z_r), \sigma(z_s)]  &=\sigma([z_r, z_s])=0,\\
    [\sigma(q_r), \sigma(z_s)]  &=[h\otimes t^r, e\otimes t^s])=~2e\otimes t^{r+s}~=2\sigma (z_{r+s})= \sigma([q_r, z_s]),\\
    [\sigma(q_r), \sigma(x_s)]  &=[h\otimes t^r, f\otimes t^s])=-2f\otimes t^{r+s}=2\sigma (x_{r+s})= \sigma([q_r, x_s]),\\
    [\sigma(z_r), \sigma(x_s)]  &=[e\otimes t^r, f\otimes t^s])=h\otimes t^{r+s}=\sigma (q_{r+s})= \sigma([z_r, x_s]).
\end{align*}
This proves result (a). The result (b) is clear.

We now prove (c). It suffices to show the result for $i=1$, since the other two cases are similar. The identity (17a) shows that $V_1$ is an irreducible $W$-module. This module is an intermediate series module because
\[
    V_1 = \bigoplus_{r\in\mathbb Z}V_{\lambda_r},
\]
where each weight space $V_{\lambda_r} = \mathbb F q_r$ of weight $-r$ is of dimension $1$ by identity (17a). From the definitions of $W$ and $V_1$, it is easily seen that $\rho_1$ is an isomorphism, so $(V_1, \rho)$ is isomorphic to the regular representation of the Witt algebra $W$.
\hfill$\Box$
\end{proof}

\section{ Representations of $A^{\delta}_{\omega}$ and ad$A^{\delta}_{\omega}$}
We first characterize the regular representation of the $3$- Lie algebra $A^{\delta}_{\omega}$.
\begin{theorem} Let $V=A$ as a vector space. The regular representation $(V,\, {\rm ad})$ of the $3$-Lie algebra $A^{\delta}_{\omega}$ is a Harish-Chandra module.
\end{theorem}
\begin{proof} Since $V=A$ as a vector space, $V=\bigoplus_{r\in\mathbb Z} (\mathbb F L_r \oplus \mathbb F M_r)$.
By Theorem 4.1 we have that
the Cartan subalgebra of $A^{\delta}_{\omega}$ is
$$
    H=\mathbb F L_0\oplus \mathbb F M_0,
$$
and
\[
    {\rm ad}(L_0, M_0)(L_t) = -tL_t \quad\text{and}\quad  {\rm ad}(L_0, M_0)(M_t) = tM_t \quad\text{for all } t\in\mathbb Z,
\]
If $t\in \mathbb Z$, define a linear function
$\lambda_t: H\wedge H\rightarrow \mathbb F$ by
\[
    \lambda_t(L_0, M_0) = -t.
\]
A simple calculation yields that $V = \bigoplus_{\lambda_t\in (H\wedge H)^*} V_{\lambda_t}$ where
\[
    V_{\lambda_t}=\mathbb F L_t\oplus \mathbb F M_{-t}
\]
are the weight spaces of dimension 2. The desired result follows. \hfill$\Box$
\end{proof}

We then describe the natural representation of the inner derivation algebra ${\rm ad}A^{\delta}_{\omega}$. Let $V$ be the vector space $A$ over $\mathbb F$. Recall that the natural module $(V, \rho)$ of ad$A^{\delta}_{\omega}$ on $V$ is defined by
\begin{equation*}\label{na}
    \rho(D)(z) = D(z) \quad\text{for all } D\in {\rm ad}A^{\delta}_{\omega}, ~z\in V.
\end{equation*}
\begin{theorem}
    The natural module $(V, \rho)$ of the derivation algebra ${\rm ad}A^{\delta}_{\omega}$ is an intermediate series module.
 \end{theorem}

\begin{proof}
From Theorem \ref{newbasis} it follows that $H=\mathbb F p_0\oplus \mathbb F q_0$ is the Cartan subalgebra of the Lie algebra ${\rm ad}A^{\delta}_{\omega}$. Under the natural action of inner derivations, for all  $t\in\mathbb \mathbb Z$, by (\ref{WLM}) and (\ref{zeros}) we obtain
\[
    p_0(L_t)=-tL_t,\quad p_0(M_t)=tM_t,\quad q_0(L_t)=-L_t,\quad\text{and}\quad q_0(M_t)=M_t.
\]
For each $t\in \mathbb Z$, define a linear map $\lambda_t: H\wedge H\rightarrow \mathbb F$ by
$\lambda_t(p_0)=t$ and $\lambda_t(q_0)=1$. Then we obtain weight spaces
\[
    V_{\lambda_t}=\mathbb F M_t \quad\text{and}\quad V_{-\lambda_t}=\mathbb F L_t.
\]
Furthermore, $V=\bigoplus\limits_{t\in \mathbb Z}(V_{\lambda_t}\oplus V_{-\lambda_t})$.
  But then $\dim V_{\lambda_t}=\dim V_{-\lambda_t}=1$ for all $t\in \mathbb Z$, so the natural module $(V, \rho)$ is an intermediate series module of ${\rm ad}A^{\delta}_{\omega}$.
\hfill$\Box$
\end{proof}

{\bf Acknowledgements} {
This work was supported in part by the National Natural Science
Foundation of China (Grant No. 11371245) and the Natural Science Foundation of Hebei
Province, China (Grant No. A2014201006).}

\vspace{5mm}
\noindent Ruipu Bai\\
College of Mathematica and Information Science\\
Hebei University\\
Baoding, Hebei, China, 071002\\
Email: bairp1@yahoo.com.cn\\
\\
Zhenheng Li ({\bf Correspondence Author})\\
College of Mathematica and Information Science\\
Hebei University\\
Baoding, Hebei, China, 071002\\
Email: zhli@hbu.edu.cn\\
\\
Weidong Wang\\
College of Mathematica and Information Science\\
Hebei University\\
Baoding, Hebei, China, 071002\\
Email: phpxx@126.cn\\


\begin{thebibliography}{www}
\bibitem{AG} D. Alexeevsky and P. Guha, On decomposability of Nambu-Poisson tensor, Acta Mathematica Universitatis Comenianae, Vol. 65 (1996) 1-9.
\bibitem{B4} R. Bai, C. Bai, J. Wang, Realizations of $3$-Lie algebras, J. Mathe. Phys., 2010, 51, 063505.
\bibitem{B2} R. Bai, W. Han, C. Bai, The generating index of an $n$-Lie algebra, J. Phys. A: Math. Theor. 44 (2011) 185201.
\bibitem{B5} R. Bai, Y. Wu, Constructions of $3$-Lie algebras, Linear and Multilinear Algebra, 2015, 63(11): 2171-2186.
\bibitem{Ba} F. Bayen, M. Flato, Remarks concerning Nambu's generalized mechanics. Phys. Rev. D 11, 3049-3053 (1975).
\bibitem{Di} L. Ding, Xiaoyu Jia, W. Zhang, On infinite-dimensional $3$-Lie algebras, Journal of Mathematical Physics, 55 (2014) 0417041.
\bibitem{Fi} V. Filippov, $n$-Lie algebras,  Sib. Mat. Zh., 26 (6), (1985) 126-140.
\bibitem{Hi} M. Hirayama, Realization of Nambu mechanics: a particle interacting with SU(2) monopole. Phys. Rev. D 16, 530-532 (1977).
\bibitem{Li}  W. Ling, On the structure of n-Lie algebras, Dissertation, University-GHS-Siegen,
\bibitem{MVV} G. Marmo, G. Visali, A. Vinogradov, The local structure of n-Poisson and n-Jacobi manifolds, arxiv.org/pdf/physics/9709046.
\bibitem{Na} Y. Nambu, Generalized Hamiltonian mechanics, Phys. Rev.,  D7, (1973) 2405-2412.
\bibitem{Mu} N. Mukunda, E. Sudarshan, Relation between Nambu and Hamiltonian mechanics. Phys. Rev. D 13, 2846-2850 (1976).
\bibitem{Po1} A. Pozhidaev, Simple quotient algebras and subalgebras of
Jacobian algebras, Sib. Math. J., 1998, 39(3), 512-517.
\bibitem{Po2} A. Pozhidaev, Monomial $n$-Lie algebras, Algebra i Logika,
1998, 37(5), 307-322.
\bibitem{Po3} A. Pozhidaev, On simple $n$-Lie algebras, Algebra i Logika,
1999, 38(3), 181-192.
Siegn, 1993.
\bibitem{Ta} L. Takhtajan, On the foundation of the generalized Nambu mechanics, Commun. Math. Phys.,  160 (1994) 295-315.
\bibitem{Vi} A. Vinogradov, M. Vinogradov, On multiple generalizations of lie algebras and poisson manifolds, American Mathematical Society,
Contemp. Math.,  219 (1998) 273- 287.
\end{thebibliography}
\end{document}